\documentclass[10pt, conference,twocolumn]{ieeeconf} 

\IEEEoverridecommandlockouts    
\usepackage{psfrag}
\usepackage{amsmath,amssymb,amsfonts,bbm}
\usepackage{latexsym}
\usepackage{graphicx}

\usepackage{colortbl}
\usepackage{fancyhdr}
\usepackage{epstopdf}
\usepackage{float}
\usepackage{hyperref}
\usepackage{booktabs}
\usepackage{enumerate}
\usepackage[ruled]{algorithm2e}
\usepackage{balance}
\usepackage{mathrsfs}

\usepackage{tikz}
\usepackage{pgfplots}
\pgfplotsset{compat=1.10}
\usepgfplotslibrary{fillbetween}

\usepackage[font=small , skip = -6pt]{caption}

\newtheorem{theorem}{Theorem}
\newtheorem{definition}{Definition}
\newtheorem{proposition}{Proposition}
\newtheorem{lemma}{Lemma}

\newtheorem{remark}{Remark}

\newtheorem{standing}{Standing Assumption}

\IEEEoverridecommandlockouts

\newcommand{\bs}{\boldsymbol}

\newcommand{\bomega}{\boldsymbol{\omega}}
\newcommand{\bx}{\boldsymbol{x}}
\newcommand{\by}{\boldsymbol{y}}

\newcommand{\bgamma}{\boldsymbol{\gamma}}
\newcommand{\cA}{\mathcal{A}}
\newcommand{\cB}{\mathcal{B}}

\newcommand{\cC}{\mathcal{C}}

\newcommand{\norm}[1]{\left\|#1\right\|}

\newcommand{\R}{\mathbb{R}}

\newcommand{\Ncal}{\mathcal{N}}

\newcommand{\mc}{\mathcal}
\newcommand{\A}{\mc{A}}

\newcommand{\dom}{\operatorname{dom}}

\newcommand{\argmin}{\operatorname{argmin}}

\newcommand{\fix}{\mathrm{fix}}
\newcommand{\proj}{\mathrm{proj}}

\newcommand{\Id}{\mathrm{Id}}
\newcommand{\Res}{\operatorname{J}}
\newcommand{\diag}{\operatorname{diag}}
\newcommand{\col}{\operatorname{col}}
\newcommand{\zer}{\operatorname{zer}}

\newcommand{\nc}{\mathrm{N}}
\newcommand{\0}{\mathbf{0}}
\newcommand{\1}{\mathbf{1}}

\newcommand{\Rmnum}[1]{\expandafter\@slowromancap\romannumeral #1@}

\makeatletter
\begin{document}

\title{A Douglas--Rachford Splitting for Semi-decentralized Equilibrium Seeking in Generalized Aggregative Games}

\author{Giuseppe Belgioioso \and Sergio Grammatico 
\thanks{G. Belgioioso is with the Control Systems group, TU Eindhoven, The Netherlands. S. Grammatico is with the Delft Center for Systems and Control (DCSC), TU Delft, The Netherlands. 
E-mail addresses: \texttt{g.belgioioso@tue.nl}, \texttt{s.grammatico@tudelft.nl}. 
This work was partially supported by NWO under research projects OMEGA (grant n. 613.001.702) and P2P-TALES (grant n. 647.003.003).
\smallskip \newline
}
}

\thispagestyle{empty}
\pagestyle{empty}

\maketitle         

\begin{abstract}
We address the generalized aggregative equilibrium seeking problem for noncooperative agents playing average aggregative games with affine coupling constraints.
First, we use operator theory to characterize the generalized aggregative equilibria of the game as the zeros of a monotone set-valued operator. Then, we massage the Douglas--Rachford splitting to solve the monotone inclusion problem and derive a single layer, semi-decentralized algorithm whose global convergence is guaranteed under mild assumptions. The potential of the proposed Douglas--Rachford algorithm is shown on a simplified resource allocation game, where we observe faster convergence with respect to forward-backward algorithms.
\end{abstract}

\section{Introduction}

\subsection*{Aggregative games in societal challenges}

An aggregative game is a collection of inter-dependent optimization problems associated with noncooperative decision makers, or agents, where each agent is affected by some aggregate effect of all the agents \cite{kukushkin:04,jensen:10}. Remarkably, aggregative games arise in several societal challenges, such as demand side management in the smart grid \cite{Saad2012}, e.g. for charging/discharging of electric vehicles \cite{ma:callaway:hiskens:13, parise:colombino:grammatico:lygeros:14, ma:zou:ran:shi:hiskens:16, grammatico:16cdc-pev}, demand-response regulation in competitive markets \cite{li2015demand}, congestion control in traffic and communication networks \cite{barrera:garcia:15}. The common denominator is in fact the presence of a large number of selfish agents, whose aggregate actions may disrupt the shared infrastructure, e.g. the power or the transportation network, if left uncontrolled.

\subsection*{Computational game theory and monotone operator theory}

Designing solution methods for multi-agent equilibrium problems in aggregative games has recently gained high research interest. A fast-growing literature has been in fact developing semi-decentralized and distributed algorithms for aggregative games without coupling constraints \cite{grammatico:parise:colombino:lygeros:16,koshal:nedic:shanbhag:16}, and semi-decentralized algorithms for aggregative games with coupling constraints \cite{grammatico:17,belgioioso:grammatico:17cdc}. With focus on the generalized Nash equilibrium (GNE) problem, the formulations in \cite{belgioioso:grammatico:17cdc, yi2017distributed} have shown an elegant approach based on monotone operator theory \cite{bauschke2017convex} to characterize the equilibrium solutions as the zeros of a monotone operator. Not only is the monotone-operator-theoretic approach general -- e.g., unlike variational inequalities, smoothness of the cost functions is not required -- but also computationally viable, since several design strategies to solve monotone inclusions are already well established, e.g.\ operator-splitting methods \cite[\S 26]{bauschke2017convex}. 


\subsection*{Critical review of solution algorithms for aggregative games}

Few algorithms are available in the literature for solving certain sub-classes of monotone aggregative games with coupling constraints, each with important technical or computational limitations. Specifically, forward-backward (FB) methods, that include projected-pseudo-gradient methods, require the pseudo-subdifferential mapping to be cocoercive \cite{belgioioso:grammatico:18ecc}, strictly \cite{koshal:nedic:shanbhag:16,liang2017distributed} or strongly monotone \cite{paccagnan:16,yi2017distributed}. To be applicable to (aggregative) games with (non-cocoercive, non-strictly) monotone pseudo-subdifferential mapping, the FB method shall be augmented with a vanishing regularization. This approach is known as iterative Tikhonov regularization (ITR) and generates a FB algorithm with double-layer vanishing step sizes \cite{kannan2012distributed}, where the actual step size shall vanish faster than the vanishing regularization term \cite[(A2.2), \S 2.1]{kannan2012distributed}. In practice, however, methods based on (double-layer) vanishing step sizes typically have slow speed of convergence, which is computationally undesirable. To solve (non-cocoercive, non-strictly) monotone aggregative games via non-vanishing iterative steps, the forward-backward-forward (FBF) method adds an additional forward step to the FB algorithm. Unfortunately, FB, ITR and FBF methods require the cost functions of the agents to have the non-differentiable part separable and dependent on the local decision variable only \cite[\S II.A]{belgioioso:grammatico:17cdc}. Recently, the preconditioned proximal-point (PPP) method was proposed to solve monotone (aggregative) games, virtually with no additional technical assumption other than monotonicity of the pseudo-subdifferential mapping \cite{yi2018distributed}. Unfortunately, however, the PPP method generates a double-layer algorithm, in which each (outer) iteration involves the solution of a sub-game without coupling constraints, via nested (inner) iterations. Together with vanishing step sizes, we can regard double-layer or nested iterations as a computational limitation.

\subsection*{Why the Douglas--Rachford splitting?}
Essentially, none of the currently available algorithms is suitable for efficiently computing a GNE in (general) monotone aggregative games.
The Douglas--Rachford (DR) splitting \cite[\S 26.3]{bauschke2017convex} has instead the potential to overcome the technical limitations of the forward-backward and forward-backward-forward methods and the computational drawbacks of the iterative-regulation and proximal-point methods.
Our motivation for studying the DR splitting is partially inspired by the celebrated Alternating Direction Method of Multipliers (ADMM) in distributed optimization, which is a special  implementation of the DR splitting method \cite{giselsson2017linear}.

\subsection*{Contribution of this paper}

Unfortunately, the direct application of the DR splitting on the classic monotone operator that defines the GNE problem does not generate a suitable solution algorithm. The second main technical difficulty is that the algorithms generated by the standard DR splitting method are not semi-decentralized. On the contrary, since we consider aggregative games, we require the structure of the computation and information exchange to be semi-decentralized, which includes the requirement that coordination signals for the agents shall be computed based on aggregate information only. 

In this paper, we focus on the class of generalized aggregative equilibria (GAE) and resolve the two complications above. Our main technical contribution is to massage the implementation of the DR splitting method with an equivalent monotone inclusion defined on an extended space. In turn, we derive a \textit{single-layer}, \textit{fixed-step}, \textit{semi-decentralized} algorithm for the computation of a GAE in aggregative games. Thanks to the semi-decentralized structure, the computational complexity of the derived algorithm is only mildly dependent (virtually, independent) on the number of agents. Finally, we show via numerical simulations that our algorithm inherits the advantages of the DR splitting in terms of fast speed of convergence.

\subsection*{Basic notation}
$\R$ denotes the set of real numbers, and $\overline{\R} := \R \cup \{\infty\}$ the set of extended real numbers. $\bs{0}$ ($\bs{1}$) denotes a matrix/vector with all elements equal to $0$ ($1$); to improve clarity, we may add the dimension of these matrices/vectors as subscript. $A \otimes B$ denotes the Kronecker product between matrices $A$ and $B$; $\left\| A \right\|$ denotes the maximum singular value of $A$; $\rm{eig}(A)$ denotes the set of eigenvalues of $A$.
Given $N$ vectors $x_1, \ldots, x_N \in \R^n$, $\boldsymbol{x} := \col\left(x_1,\ldots,x_N\right) = \left[ x_1^\top, \ldots , x_N^\top \right]^\top$.

\subsection*{Operator theoretic definitions}
$\Id(\cdot)$ denotes the identity operator. The mapping $\iota_{S}:\R^n \rightarrow \{ 0, \, \infty \}$ denotes the indicator function for the set $\mc{S} \subseteq \R^n$, i.e., $\iota_{S}(x) = 0$ if $x \in S$, $\infty$ otherwise. For a closed set $S \subseteq \R^n$, the mapping $\proj_{S}:\R^n \rightarrow S$ denotes the projection onto $S$, i.e., $\proj_{S}(x) = \argmin_{y \in S} \left\| y - x\right\|$. The set-valued mapping $\nc_{S}: \R^n \rightrightarrows \R^n$ denotes the normal cone operator for the the set $S \subseteq \R^n$, i.e., 
$\nc_{S}(x) = \varnothing$ if $x \notin S$, $\left\{ v \in \R^n \mid \sup_{z \in S} \, v^\top (z-x) \leq 0  \right\}$ otherwise.
For a function $\psi: \R^n \rightarrow \overline{\R}$, $\dom(\psi) := \{x \in \R^n \mid \psi(x) < \infty\}$; $\partial \psi: \dom(\psi) \rightrightarrows {\R}^n$ denotes its subdifferential set-valued mapping, defined as $\partial \psi(x) := \{ v \in \R^n \mid \psi(z) \geq \psi(x) + v^\top (z-x)  \textup{ for all } z \in {\rm dom}(\psi) \}$;
A set-valued mapping $\mathcal{F} : \R^n \rightrightarrows \R^n$ is $\ell$-Lipschitz continuous, with $\ell>0$, if $\|u-v\| \leq \ell\|x-y\|$ for all $x,y \in \R^n$, $u \in \mathcal{F} (x)$, $v \in \mathcal{F} (y)$; $\mathcal{F} $ is (strictly) monotone if $(u-v)^\top  (x-y) \geq (>) \, 0$ for all $x \neq y \in \R^n$, $u \in \mathcal{F} (x)$, $v \in \mathcal{F} (y)$; $\mathcal{F} $ is $\eta$-strongly monotone, with $\eta>0$, if 
$(u-v)^\top (x-y) \geq \eta \left\| x-y \right\|^2$ for all $x \neq y \in \R^n$, $u \in \mathcal{F} (x)$, $v \in \mathcal{F} (y)$; $\mathcal{F} $ is $\eta$-{averaged}, with $\eta \in (0,1)$, if  
$\left\| \mathcal{F} (x) - \mathcal{F} (y) \right\|^2 \leq \left\| x-y \right\|^2 - \tfrac{1-\eta}{\eta}\left\| \left( \textup{Id}-\mathcal{F}  \right)(x) - \left(\textup{Id}-\mathcal{F}  \right)(y) \right\|^2$, for all $x, y \in \R^n$;
$\mathcal{F} $ is $\beta$-cocoercive, with $\beta>0$, if $\beta \mathcal{F} $ is $\tfrac{1}{2}$-averaged.
With ${\rm J}_{\mathcal{F} }:=(\Id + \mathcal{F} )^{-1}$, we denote the resolvent operator of $\mathcal{F} $, which is $\tfrac{1}{2}$-averaged if and only if $\mathcal{F} $ is monotone; $\fix\left( \mathcal{F}\right) := \left\{ x \in \R^n \mid x \in \mathcal{F}(x) \right\}$ and $\zer\left( \mathcal{F}\right) := \left\{ x \in \R^n \mid 0 \in \mathcal{A}(x) \right\}$ denote the set of fixed points and of zeros, respectively.

\section{Generalized aggregative games} \label{sec:GAG}

\subsection{Mathematical formulation}
We consider a set of $N$ noncooperative agents, where each agent $i \in \Ncal := \{1,\ldots, N \}$ shall choose its decision variable (i.e., strategy) $x_i$ from the local decision set $\Omega_i \subseteq \mathbb{R}^n$ with the aim of minimizing its local cost function $\left( x_i, \bs{x}_{-i} \right) \mapsto J_i\left( x_i, \bs{x}_{-i} \right): \R^n \times \R^{n(N-1)} \rightarrow \overline{\R}$, which depends on both the local variable $x_i$ (first argument) and on the decision variables of the other agents, $\bs{x}_{-i} = \col\left( \{ x_j \}_{j \neq i} \right)$ (second argument).

We focus on the class of average aggregative games, where the cost function of each agent depends on the local decision variable and on the value of the average strategy, i.e.,
\begin{equation*} 
\textstyle
\hat{x}:= \frac{1}{N}\sum_{i=1}^{N} x_i = M_n \bs x,
\end{equation*}
where $M_n := \frac{1}{N}\mathbf{1}^\top_N \otimes I_n$ and $\bs x = \col(x_1, \cdots , x_N)$.
Thus, for each $i \in \Ncal$, there exists a function $f_i: \R^n \times \R^n \rightarrow \overline{\R}$ such that cost function $J^i$ can be written as
\begin{align*} 
\textstyle
J_i(x_i, \bs{x}_{-i}) &=: \textstyle  f_i \left( x_i, M_n \bs x \right)
\end{align*}

Furthermore, we consider \textit{generalized games}, where the coupling among the agents arises not only via the cost functions, but also via their feasible decision sets. In our setup, the coupling constraints are described by an affine function, $\bs x \mapsto A \bs x - b$, where $A \in \R^{m \times nN}$ and $b \in \R^m$. Thus, the collective feasible set, $\bs{\mc X} \subseteq \R^{nN}$, reads as
\begin{equation}
\label{eq:G}
\bs{\mc X} := \prod_{i=1}^N \Omega_i  \bigcap \left\{ \bs x \in \R^{Nn} | \, A \bs x - b \leq \bs{0}_m \right\};
\end{equation}
while the feasible decision set of each agent $i \in \mc N$ is characterized by the set-valued mapping $\mc{X}_i$,
defined as
$$ \textstyle
\mc{X}_i(\bs{x}_{-i}) := \big\{ y_i \in \Omega_i \, | \, A_i y_i -b_i\leq  \sum_{j \neq i}^N b_j-A_j x_j \big\},
$$
where $\sum_{i=1}^N b_i = b$, $A_i \in \R^{m \times n}$ and $A = \left[ A_1, \ldots, A_N \right]$.
The set $\Omega_i$ represents the local decision set for agent $i$, while the matrix $A_i$ and $b_i$ are local data which characterize how agent $i$ is involved in the coupling constraints.
\begin{remark}[Affine coupling constraint]
Affine coupling constraints as considered in this paper are very common in the literature of noncooperative games, see \cite{paccagnan:16},
\cite{grammatico:17}, \cite{yi2017distributed}, \cite{liang2017distributed}.
Moreover, we recall that the more general case with separable convex coupling constraints can be reformulated as game via affine coupling constraints \cite[Remark 2]{grammatico:17tcns}.
{\hfill $\square$}
\end{remark}

\smallskip
Next, we postulate standard convexity and compactness assumptions for the constraint sets, and convexity of the cost functions with respect to their local decision variable.

\smallskip
\begin{standing}[Compact convex constraints] \label{ass:CS}
For  each $i \in \mc N$, the set $\Omega_i$ is nonempty, compact and convex. The set $\bs{\mc X}$ satisfies the Slater's constraint qualification.
{\hfill $\square$}
\end{standing}
\smallskip

\begin{standing}[Convex functions] \label{ass:CCF}
For all $i\in \Ncal$, and for all fixed $\boldsymbol{z} \in \prod_{j\neq i} \Omega_j$ and $w \in \frac{1}{N} \sum_{j\neq i} \Omega_j$
the functions $J_i\left( \, \cdot \, , \, \boldsymbol{z} \right)$ and $f_i\left( \, \cdot \, , \, w\right)$ are convex.
{\hfill $\square$}
\end{standing} 

\smallskip

In summary, the aim of each agent $i$, given the decision variables of the other agents, is to choose a strategy $x_i$ that solves its local optimization problem, according to the game setup previously described, i.e., 
\begin{align}\label{eq:Game}
\underset{x_i \in \, \R^n}{\min} \; 
J_i \big( x_i, \bs{x}_{-i} \big) \quad \text{s.t. } x_i \in \mc X_{i}(\bs x_{-i}).
\end{align}

\subsection{Nash vs aggregative equilibria}
From a game-theoretic perspective, we consider the problem to compute a Nash equilibrium, as formalized next.

\smallskip
\begin{definition}[Generalized $\varepsilon$-Nash equilibrium]
A collective strategy $\bs x^{*}$ is a generalized $\varepsilon$-Nash equilibrium ($\varepsilon$-GNE) of the aggregative game in \eqref{eq:Game} if $\bs x^* \in \bs{\mc X}$ and, for all $i\in \mc N$ and for all $z \in \mc X_i(\bs x^*_{-i})$
\begin{align} \textstyle \label{eq:eps-GNE}
f_i\left( x^*_{i}, M_n \bs x^* \right) \leq 
\ f_{i} \left( z, \, \frac{1}{N}z+ \frac{1}{N}\sum_{j\neq i}^N x_j^* \right) + \varepsilon.
\end{align}
If \eqref{eq:eps-GNE} holds with $\varepsilon=0$ then $\bs x^*$ is a GNE.{
\hfill $\square$}
\end{definition}
\smallskip

In other words, a set of strategies is a genralized Nash equilibrium if no agent can improve its objective function by unilaterally changing its strategy to another feasible one.

\smallskip
The concept of aggregative equilibrium springs from the intuition that the contribution of each agent to the aggregation decreases as the population size grows. Technically, at the limit for $N \rightarrow \infty$, the decision variable of agent $i$ does not influence the second argument of its cost function $f_i$.

\smallskip
\begin{definition}[Generalized aggregative equilibrium]
A collective strategy $\bar{ \bs x}$ is a generalized aggregative equilibrium (GAE) of the aggregative game in \eqref{eq:Game} if $\bar{\bs x} \in \bs{\mc X}$ and, for all $i\in \mc N$ and for all $z \in \mc X_i(\bar{\bs x}_{-i})$
\begin{align*} \textstyle 
f_i\left( \bar{x}_{i}, M_n \bar{\bs x} \right) \leq 
\ f_{i} \left( z, \, M_n \bar{\bs x} \right).
\end{align*}
\vspace*{-3em}

{\hfill $\square$}

\vspace*{.5em}
\end{definition}
Nash and aggregative equilibria are strictly connected. Indeed, under some mild assumptions it can be proven that every GAE equilibrium is an $\varepsilon$-GNE equilibrium, with $\epsilon$ tending to zero as $N$ grows \cite[\S 4]{paccagnan2018nash}.

Under Assumptions \ref{ass:CS}$-$\ref{ass:CCF}, the existence of a GNE and a GAE of the game in \eqref{eq:Game} follows from Brouwer's fixed-point theorem \cite[Prop.~12.7]{palomar2010convex}, while uniqueness does not hold in general.

\subsection{Variational equilibria and pseudo-subdifferentials} \label{subsec:VA}
Within all the possible equilibria, we focus on an important subclass, with some relevant structural properties, such as ``larger social stability'' and ``economic fairness" \cite[Th. 4.8]{facchinei2010generalized}, that corresponds to the solution set of an appropriate generalized variational inequality\footnote{\textit{Definition (Generalized variational inequality):}
Consider a closed convex set $ W \subseteq \R^n$, 
a set-valued mapping $\Psi: W \rightrightarrows \R^n$ and a single-valued mapping $\psi:  W \rightarrow \R^n$. 
The generalized variational inequality problem GVI$( W,\Psi)$, is the problem to find $x^* \in  W$ and $g^* \in \Psi(x^*)$ such that
$ \textstyle
(x-x^*)^\top \, g^* \geq 0 \textup{ for all } x\in  W
$. If $\Psi (x) = \{\psi(x)\}$ for all $x\in  W$, then GVI$(\mc W,\Psi)$ reduces to the variational inequality problem VI$( W,\psi)$.
{\hfill $\square$}
}%
 (GVI).

A fundamental mapping in noncooperative games is the so-called \textit{pseudo-subdifferential}, $ F: \bs{\mc X} \rightrightarrows \R^{nN}$, defined as
\begin{align} 
\label{eq:PsGr}
 F(\bs x) &:= 
\col\left( \left\{  \partial_{x_i} \, f_i \left( x_i, \, M_n \bs x \right) \right\}_{i \in \mathcal{N}} \right). 
\end{align}
Namely, the mapping $ F$ is obtained by stacking together the subdifferentials of the agents' objective functions with respect to their local decision variables.

Under Assumptions \ref{ass:CS}$-$\ref{ass:CCF}, it follows by \cite[Prop. 12.4]{palomar2010convex} that any solution to GVI$(\bs{\mc X}, F)$ is a (\textit{variational}) generalized Nash equilibrium (v-GNE) of the game in \eqref{eq:Game}. The inverse implication is not true in general, and actually in passing from the Nash equilibrium problem to the GVI, most solutions are lost \cite[\S 12.2.2]{palomar2010convex} -- indeed, a game may have a Nash equilibrium while the corresponding GVI has no solution. 
Note that, if the cost functions $J_i$'s are differentiable, then GVI$(\bs{\mc X}, F)$ reduces to VI$(\bs{\mc X}, F)$, which can be solved via projected-pseudogradient algorithms \cite{koshal:nedic:shanbhag:16,paccagnan:16,belgioioso2017convexity}, \cite{facchinei:pang}.

Similarly, given the mapping $F_{\text a}: \bs{\mc X} \rightrightarrows \R^{nN}$, defined as
\begin{align} 
\label{eq:PsGrAgg}
 F_\text{a}(\bs x) &:= 
\col\left( \left\{  \partial_{x_i} \, f_i \left( x_i, \, z  \right)|_{z = M_n \bs x} \right\}_{i \in \mathcal{N}} \right),
\end{align}
one can prove that every solution to GVI$(\bs{\mc X}, F_\text{a})$ is a (\textit{variational}) generalized aggregative equilibrium (v-GAE) of the game in \eqref{eq:Game}.
The remainder of the paper is devoted to design a semi-decentralized algorithm to compute a v-GAE.

\section{Generalized aggregative equilibrium as zero of the sum of two monotone operators} \label{sec:GNE-MO}

In this section, we exploit operator theory to recast the GAE seeking problem into a monotone inclusion, namely, the problem of finding a zero of a set-valued monotone operator. 

To make the resulting monotone inclusion suitable for the application of the DR splitting, we extend the original game in \eqref{eq:Game}, by including an additional player. Specifically, let us introduce the \textit{extended game} characterized by the following $N+1$ coupled optimization problems: 
\begin{align} \label{eq:EG_1}
(\forall i \in \mathcal{N})\; & 
\begin{cases}
\underset{
x_i, \, y_i 
}{\argmin} & f_i \left(x_i,\sigma \right) + \iota_{C_i}(x_i,y_i) 
\\
\hspace*{2em} \text{s.t.} 
&  M_m \bs y  \leq \0_m \\
& \sigma - M_n \bx = \0_n
\end{cases} \qquad\\
\label{eq:EG_2}
&
\begin{cases}
\underset{\sigma \in \R^n}{\argmin}& f_{\text c}(\bs x, \sigma)\\
\hspace*{2em} \text{s.t.} & \sigma - M_n \bx = \0_n
\end{cases}
\end{align}
where  $C_i := \{ (x_i,y_i) \in \Omega_i \times \R^m \, |\; y_i = A_i x_i -b_i \}$ is a local constraint set and $f_{\text c}(\bs x, \sigma) :=0$. 

\smallskip
The next statement shows that the v-GNE of the extended game in \eqref{eq:EG_1}-\eqref{eq:EG_2} fully characterize the v-GAE of the original game in \eqref{eq:Game}.

\begin{proposition} \label{prop:games}
The collective strategy $\bar{\bx}$ is a v-GAE of the game in \eqref{eq:Game} if and only if $\col(\bar{\bx},\bar{\by},\bar{\sigma})$ is a v-GNE of the game in \eqref{eq:EG_1}--\eqref{eq:EG_2}, with $\bar{\sigma} = M_n\bar{\bs x}$ and $\bar{y}_i = A_i \bar{x}_i-b_i$ for all $i \in \mathcal{N}$.
{\hfill $\square$}
\end{proposition}


\smallskip
\begin{remark} 
According to a semi-decentralized communication structure, the additional player in \eqref{eq:EG_2} represents the central coordinator, which does not participate in the game, i.e., $f_{\text c}(\bs x, \sigma) \equiv 0$, and whose ``decision variable" $\sigma$ must be equal to the average strategy, i.e., $\sigma = M_n \bs x$.
{\hfill $\square$}
\end{remark}

\smallskip
\begin{remark} 
The additional local variables $y_i$'s and constraint sets $ C_i$'s transform the original affine coupling constraints, i.e., $A \bx - b \leq 0$, into: (1) $N$ local constraints, i.e., $(x_i,y_i)\in C_i$, and (2) one coupling constraint in aggregative form, $ M_m \bs y  \leq 0$, which is more convenient to handle in a semi-decentralized communication structure.
{\hfill $\square$}
\end{remark}

\smallskip
Next, we show that the v-GAE of the original game in \eqref{eq:Game} are zeros of a set-valued operator obtained by grouping the KKT conditions of the extended game in \eqref{eq:EG_1}--\eqref{eq:EG_2}.  

With this aim, let us introduce the mapping $\mc T: \R^d\rightrightarrows \R^d$, with $d=nN+mN+2n+m$, defined as
\begin{align} \label{eq:T}
\mathcal{T}&: \begin{bmatrix}
\bx \\ \by \\ \sigma \\ \mu \\ \lambda
\end{bmatrix} \mapsto
\begin{bmatrix}
F_{\text e}(\bx,\sigma) + \partial_{\bx}\iota_{\boldsymbol{{C}}}(\bx,\by) -M_n^\top \mu \\
\partial_{\by}\iota_{\bs{{C}}}(\bx,\by) + M_m^\top \lambda\\
\partial_\sigma f_{\text c}(\bs x, \sigma) + \mu \\
-(\sigma - M_n \bx)
\\
\nc_{\R^m_{\geq 0 }}(\lambda) - M_m \by 
\end{bmatrix} ,
\end{align}
where $\bs{{C}}:=\{ (\bs x,\bs y)\in \R^{(n+m)N} \, |\, (x_i,y_i)\in C_i,\,  \forall i \in \mathcal{N} \}$, $\partial_\sigma f_{\text c}= \0$ and $F_{\text e}: \R^{nN+n} \rightrightarrows \R^{Nn}$ is defined as
\begin{equation*}
F_{\text e}(\bx,\sigma) = \col 
\left(
\left\{
\partial_{x_i} f_i \left(
x_i,\sigma
\right)
\right\}_{i=1}^N
\right).
\end{equation*}

\begin{proposition} \label{prop:opT-vGNE}
The following statements are equivalent:
\begin{enumerate}[(i)]
\item $\bx^*$ is a v-GNE	of the extended game in \eqref{eq:EG_1}--\eqref{eq:EG_2};
\item $\exists \lambda^* \in \R^{m}_{\geq 0 }$ such that $\col(\bx^*,\by^*,\sigma^*,\mu^*,\lambda^*) \in \zer \mathcal{T}$, with $\sigma^* = M \bx^*$, $\mu^* = 0$ and $y_i^* = A_i x_i^*$,  $\forall i \in \mathcal{N}$.
{\hfill $\square$}
\end{enumerate}
\end{proposition}

\smallskip

Now, we show that the set-valued operator $\mathcal{T}$ can be written as the sum of two mappings that are maximally monotone if the \textit{extended pseudo-subdifferential} $F_\text e \times \partial_\sigma f_{\text c}$ is such. Thus, let us introduce the mappings
\begin{align} \label{eq:A} 
\mathcal{A}&:=( F_{\text e} \times \0_{d_1}) +(\nc_{\bs C} \times \0_{d_2}),\\
\label{eq:B}
\mc B &:= (\0_{d_3} \times \nc_{\R^m_{\geq 0}})+S,
\end{align}
where $d_1=d-nN$, $d_2=d_1-mN$, $d_3=d-m$ and
\begin{align*} \textstyle
S:\begin{bmatrix}
\bx \\ \by \\ \sigma \\ \mu \\ \lambda
\end{bmatrix} \mapsto
\begin{bmatrix}
0 & 0 & 0 & -M_n^\top &  0 \\
0 & 0 & 0 &  0 & M_m^\top  \\
0 & 0 & 0 &  I & 0 \\
M_n & 0 & -I & 0 & 0 \\
0 & -M_m & 0 & 0 & 0
\end{bmatrix}
\begin{bmatrix} \textstyle
\bx \\ \by \\ \sigma \\ \mu \\ \lambda
\end{bmatrix}.
\end{align*}

\smallskip
%

\begin{standing}[Extended monotonicity]
\label{ass:E-MON}
The mapping $F_\text e \times \partial_\sigma f_{\text c}$ is maximally monotone.
{\hfill $\square$}
\end{standing}

\smallskip

\begin{lemma} \label{lemm:Splitting}
The mapping $\mathcal{T}$ in \eqref{eq:T} can be split as $\mathcal{T} = \mathcal{A}+\mathcal{B}$, with $\mathcal{A}$ and $\mathcal{B}$ as in \eqref{eq:A}-\eqref{eq:B}. If Assumption \ref{ass:E-MON} holds, then the mappings $\cA$ and $\cB$ are maximally monotone.
{\hfill $\square$}
\end{lemma}

\smallskip
\begin{remark}
Strong (strict) monotonicity of the pseudo-subdifferential $F$ in \eqref{eq:PsGr} is a usual assumption in the literature of noncooperative game theory, see \cite{yi2017distributed}, \cite{paccagnan:16}, \cite{koshal:nedic:shanbhag:16}, \cite{belgioioso:grammatico:17cdc}. Here, we postulate (non-strict) monotonicity of the extended pseudo-subdifferential $F_\text e \times \partial_\sigma f_{\text c}$. Indeed, Assumption \ref{ass:E-MON}, represents an extended monotonicity assumption for the augmented spaces, see \cite[A. 3]{salehisadaghiani2017distributed} for a similar assumption.

{\hfill $\square$} 
\end{remark}

\smallskip
\begin{remark} \label{rem:Existence}
Since $ F_{\text e}(\bs x, \sigma)|_{\sigma = M \bs x}=  F_{\text{a}}(\bs x)$, Assumption \ref{ass:E-MON} implies the monotonicity of the mapping $ F_{\text a} $ in \eqref{eq:PsGrAgg}. By \cite[Prop.~12.11]{palomar2010convex}, this is a sufficient condition for the existence of a v-GAE of the game in \eqref{eq:Game}.
{\hfill $\square$}
\end{remark}

\section{Semi-decentralized aggregative equilibrium seeking via Douglas-Rachford splitting} \label{sec.DRalg}

In this section, we derive a single layer, fixed step, semi-decentralized algorithm
to compute a v-GAE in aggregative games as in \eqref{eq:Game}. The algorithm is obtained by solving the monotone inclusion in Prop. \ref{prop:opT-vGNE}(ii) via DR splitting.

\subsection{Douglas--Rachford operator splitting}
In Section \ref{sec:GNE-MO} we show that the original GAE equilibrium problem is fully characterized by the monotone inclusion
\begin{equation} \label{eq:MonINc}
\bs \omega^* \in \zer(\mc A + \mc B),
\end{equation}
where $\mc A$ and $\mc B$ are maximally monotone mappings. 

To solve \eqref{eq:MonINc}, several operator splittings can be considered \cite[\S 26]{bauschke2017convex}. Here, we adopt the DR splitting \cite[Th. 26.11]{bauschke2017convex}, whose iterations for the mappings $\mc A$ and $\mc B$ are recalled next.

Let $(\lambda_n)_{n\in\mathbb{N}}$ be a sequence in $[0,2]$ such that $\sum_{n \in \mathbb{N}} \lambda_n (2-\lambda_n) = + \infty $. Set $\bs \omega^0 \in \R^d$, then
\begin{align} \label{eq:DRsplit}
(\forall k \in \mathbb{N}) \quad
\left\lfloor
\begin{array}{l}
\bomega^{k+1/2} = \Res_{ \cA }( \tilde{\bomega}^k), \\
\tilde{\bomega}^{k+1/2} = 2 \bomega^{k+1/2}- \tilde{\bomega}^k,\\
\bomega^{k+1} = \Res_{\cB}(\tilde{\bomega}^{k+1/2}), \\
\tilde{\bomega}^{k+1} = \tilde{\bomega}^k +\lambda_k (\bomega^{k+1} -\bomega^{k+1/2}).
\end{array} \right.
\end{align}

The DR method demands for the solution of some implicit equations to evaluate the resolvent operators $J_{ \mc A}$ and $J_{\mc B}$ (see Lemma \ref{lem:resolvents} in the Appendix). On the contrary, non-implicit forward-backward methods perform explicit evaluation of either $\mc A$ or $\mc B$, at the cost of conditional stability in the form of step size constraints. However, the advantage of unconstrained step sizes comes with the price that the implicit equations can be solved only approximately in practice.

In the next subsection, we discuss the semi-decentralized algorithm obtained by explicitly writing the iterations in \eqref{eq:DRsplit}.
The formal derivation of the algorithm is in the Appendix.

\subsection{Semi-decentralized algorithm}

Our DR algorithm works as follows: The central coordinator is responsible for the updates of the multipliers $\lambda$, $\mu$ and the strategy $\sigma$, and communicates with the agents via broadcasting. The agents update their local variables $x_i$'s, $y_i$'s and communicate with the central coordinator only in aggregative form.

The next table summarizes the proposed DR algorithm.
\begin{center}
\begin{minipage}[b!]{\columnwidth}
\hrule
\smallskip
\textsc{Algorithm 1}: Semi-decentralized Douglas--Rachford
\smallskip
\hrule
\medskip

\textbf{Initialization}:
\begin{enumerate}[1.]
\item Local: For all $i \in \mc N$, set $\gamma_i >0$, $x^0_i\in \Omega_i$, $y_i^0 = A_i x_i^0-b_i$. 

\smallskip
\item Central: Set $\sigma^0=M_n \bs x^0$, $\mu^0 = \0_n$, $\lambda^0 \in \R^m_{\geq 0}$, $\alpha>0$, $\delta_\text{c} \in (0, \frac{1}{\hat{\gamma}})$, $\beta_\text{c} \in (0, \frac{1}{\alpha + \hat{\gamma}/N})$, where $\hat{\gamma} := \frac{1}{N} \sum_{i=1}^N \gamma_i$.
\end{enumerate}

\medskip
\noindent
\textbf{Iterate until convergence}:

\smallskip
\begin{enumerate}[1.]
\item a) Local strategy update: for each agent $i \in \mc N$
\begin{align*}
\textstyle
x_i^{k+1} &= \textstyle \underset{z \in {\Omega}_i}{\argmin} \, 
f_i \left( z,\sigma^k \right)
+{(A_i^\top \lambda^k - \frac{1}{N} \mu^k)}^\top z
\\[-.5em]
& \hspace*{7em} + \tfrac{1}{2 \gamma_i} \norm{z-x_i^k}_{(I+A_i^\top A_i)}^2
, \\
y_i^{k+1} &= \textstyle A_i x_i^{k+1}-b_i.
\end{align*}

b) Communication: $(\hat x^{k+1}, \hat y^{k+1}) \rightarrow$ coordinator.

\medskip
\item a) Central multipliers and strategy updates: \noindent
\begin{align*}
\textstyle
\lambda^{k+1} &= \textstyle \, {\proj}_{\R^m_{\geq 0}} \,  \big(
\lambda^{k}  + \delta_\text{c} \big( 2\hat{y}^{k+1}- \hat{y}^k  \big)
\big),\\
\textstyle
\mu^{k+1} &= \textstyle  \mu^k - \beta_\text{c} \, (2\hat{x}^{k+1}- \hat{x}^k  - \sigma^{k} + \alpha \mu^k ), \\
\textstyle 
\sigma^{k+1} &= \textstyle \sigma^{k}- \alpha \mu^{k+1}.
\end{align*}
\smallskip
b) Broadcast: $(\lambda^{k+1}, \mu^{k+1},\sigma^{k+1}) \rightarrow$ agents.
\end{enumerate}
\hrule
\end{minipage}
\end{center}

%

\smallskip
\noindent
\textbf{Communications}: At each iteration $k+1$, a communication round between the agents and the central coordinator takes place. Specifically, after a local update, the agents forward their updated strategies, $x_i^{k+1}$'s, $y_i^{k+1}$'s, to the central coordinator, which receives this information in aggregative form, i.e., $\hat x^{k+1} =\frac{1}{N}\sum_{j=1}^N x_j^{k+1} $, $\hat y^{k+1}=\frac{1}{N}\sum_{j=1}^N y_j^{k+1}$. In turn, the coordinator broadcasts the updated multipliers, $\lambda^{k+1}$, $\mu^{k+1}$ and its strategy $\sigma^{k+1}$, to the agents. 

\smallskip
\noindent
\textbf{Strategies update}:
At iteration $k+1$, for all $i \in \mc N$, the agent $i$ updates its local strategies as follows:
\begin{align}
\textstyle \nonumber
x_i^{k+1} &= \textstyle \underset{z \in \bs{\Omega}_i}{\argmin} \,  f_i \left( z,\sigma^k \right)
+\langle A_i^\top \lambda^k - \frac{1}{N}\mu^k, \, z \rangle
\\[-.5em] \label{eq:localUp}
& \hspace*{6.5em} + \tfrac{1}{2 \gamma_i} \norm{z-x_i^k}_{(I+A_i^\top A_i)}^2,\\
y_i^{k+1} &= \textstyle A_i x_i^{k+1}-b_i. \nonumber
\end{align}
The term ${(A_i^\top \lambda^k - \frac{1}{N} \mu^k)}^\top z$ is a penalization to satisfy the coupling constraints $A \bs x -b \leq \0_m$ and $\sigma - M \bs x = \0_n$, while $\tfrac{1}{2 \gamma_i} \norm{z-x_i^k}_{(I+A_i^\top A_i)}^2$ is a weighted proximal term.

\smallskip
After the communication step that follows the local actions update \eqref{eq:localUp}, the central coordinator updates the multipliers and its strategy as follows:
\begin{align}
\textstyle \label{eq:centrLam}
\lambda^{k+1} &= \textstyle \, {\proj}_{\R^m_{\geq 0}} \,  \big(
\lambda^{k}  + \delta_\text{c} \big( 2\hat{y}^{k+1}- \hat{y}^k \big)
\big),\\
\textstyle \label{eq:centrMu}
\mu^{k+1} &= \textstyle  \mu^k - \beta_\text{c} \, (2\hat{x}^{k+1}- \hat {x}^k  - \sigma^{k} + \alpha \mu_k ), \\
\textstyle \label{eq:centrSig}
\sigma^{k+1} &= \textstyle \sigma^{k}- \alpha \mu^{k+1}.
\end{align}
Remarkably, \eqref{eq:centrLam} coincides with the dual update of projected pseudo-gradient methods for generalized games \cite{yi2017distributed}, \cite{belgioioso:grammatico:18ecc}, \cite{paccagnan2018nash}. On the other hand, the updates \eqref{eq:centrMu}-\eqref{eq:centrSig} make sure that the variable $\sigma$ tracks the average state of the population, $\hat{x}$, and asymptotically converges to it.
We remark that all the central updates \eqref{eq:centrLam}--\eqref{eq:centrSig} require information in aggregative form, i.e., $\hat x^{k+1}$, $\hat y^{k+1}$, only.

\smallskip
In the next statement, we establish the global convergence of Algorithm 1 to a v-GAE of the game in \eqref{eq:Game}.

\smallskip
\begin{theorem}
\label{th.Convergence}
The sequence $\left( \col(\bs x^k, \bs y^k, \sigma^k, \mu^k, \lambda^k)    \right)_{k=0}^\infty$ generated by Algorithm 1
globally converges to some $ \col(\bs x^*, \bs y^*, \sigma^*, \mu^*, \lambda^*)$ $\in \zer( \mc T)$, with $\mc T$ as in \eqref{eq:T}, where $\bs x^*$ is a v-GAE of the game in \eqref{eq:Game}.
{\hfill $\square$}
\end{theorem}

\section{Resource allocation game} \label{sec.IFBS}
\subsection{Illustrative problem setup}
For illustration purposes, we consider a simplified resource allocation game. Specifically, we suppose that each agent has to complete a given task in $n$ time slots. Let us denote by $x_i(h) \in [0, 1]$ the ratio of the task that agent $i$ allocates at time slot $h$, and by $\bar{u}_i\in [0, 1]^n$ a personalized vector of upper constraints on the maximum allowed allocation for each time slot. The set of possible allocation vectors for agent $i$ is therefore given by
\begin{equation} \label{eq:lc_ex} \textstyle
\Omega_i := \{ x_i \in \R^n | \; \mathbf{0} \leq x_i \leq \bar{u}_i, \;\mathbf{1}^\top_n x_i = 1 \}, 
\end{equation} 
with component-wise inequalities. Moreover, we suppose that there exists an upper constraint $\bar b(h) >0$, for the sum among the local allocations, at each time slot $h$, i.e., 
\begin{equation*} \textstyle
\sum_{i=1}^N w_i x_i(h) \leq \bar b(h), \quad \text{for }  h = 1,\cdots,n
\end{equation*}
where $w_i \geq 0$ weights the contribution of agent $i$ in the summation. The admissible overall strategy profile $\boldsymbol{x}$, must therefore lie within the set
\begin{equation} \label{eq:constrSim} \textstyle
\bs{\mc{X}} :=\{ \bs{x} \in \prod_{i=1}^N \Omega_i \, | \, A\bs{x} \leq b \},
\end{equation}
where $b:= \col \left(\{ \bar{b}(1), \cdots, \bar{b}(n) \right)$ and $A := w^\top \otimes I_n$, with $w := \col (w_1,\cdots, w_N)$.

The aim of each agent $i$ is to choose a strategy $x_i$ within \eqref{eq:lc_ex} and such that $(x_i,\bs{x}_{-i})$ satisfies \eqref{eq:constrSim}, while minimizing its local cost function $J_i$, defined as
\begin{align} \label{eq:CFsim} \textstyle
J_i(x_i, \bs{x}_{-i}) &:= \frac{1}{2} a_i \norm{x_i - \tilde{x}_i} ^2 + { \left( Q_i \, M \bs{x} \right)}^\top x_i,
\end{align}
where $a_i \in \R_{>0}$ and  $Q_i\succeq 0$.
The cost function in \eqref{eq:CFsim} represents the cost for deviating from a pre-fixed allocation schedule $\tilde{x}_i$ plus an additional price that is proportional to the average allocation, $M \bs{x}$. Essentially, each agent has an incentive to allocate its task in time slots that are not congested by the other agents.

\subsection{Numerical study}
We consider $N=10^3$ agents with $n=10$ time slots and we randomly set their parameters as $a_i \sim [1, 2]$, $w_i \sim [1, 2]$, $ Q_i = q_i I_n + \bar{Q}_i$, with $q_i \sim [1, 2] $, $(\bar{Q}_i)_{j,k} \sim [0, 0.1]$, for all $j,k \in \mc N$, all with uniform distribution. We randomly generate the local and the coupling bounds as:
$\bs{1}^\top \bar{u}_i = 2$ and $b_i=\frac{b}{N}$, for all $i \in \mc N$ , with $\tfrac{1}{2} A \bar{ \bs u} \leq b \leq \tfrac{2}{3} A\bar{ \bs u}$
to guarantee the feasibility and avoid redundancy in the constraints. The desired allocation vectors are set, for all $i \in \mc{N}$, as $\tilde{x}_i := \proj_{\mc{X}_i} (e_1)$, where $e_1 := [1,0, \cdots, 0]^\top \in \R^n$, namely, each agent wants to complete its task in the first time slot. 

In Fig. \ref{fig:1}, we confront the convergence properties on the described scenario of the proposed DR algorithm versus the \textit{preconditioned Forward-Backward} (pFB) algorithm in \cite[Alg. 1]{belgioioso:grammatico:18ecc}, modified to find a v-GAE.
Specifically, to evaluate the convergence speed of the algorithms, we consider the sequence $\norm{\bs x^{k}-\bar{ \bs x}}$, i.e., the distance of the estimated solution at iteration $k$, $\bs x^k$, from the v-GAE of the game, $\bar{\bs x}$.

Fig. \ref{fig:1} shows the mean value of the sequence $\norm{\bs x^{k}-\bar{ \bs x}}$ in 50 simulations,
for the DR and the pFB algorithms. We note that, on this scenario, the DR algorithm converges, in average, more than 10 times faster with respect to the projection-type algorithm.

\begin{figure}
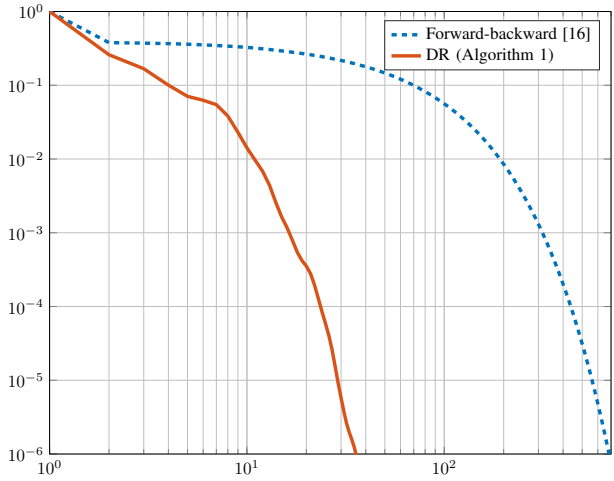

%
%
\definecolor{mycolor1}{rgb}{0.00000,0.44700,0.74100}%
\definecolor{mycolor2}{rgb}{0.85000,0.32500,0.09800}%
%
\caption{Mean value of the sequence $\norm{\bs x(k)-\bar{\bs x}}/\norm{x(0)-\bar{\bs x}}$ in 50 simulations, for the DR algorithm (red line) and the  preconditioned forward-backward (blue line) in \cite{belgioioso:grammatico:18ecc}. The parameters of the DR are set as follows: $\alpha = 1 $, $\delta_\text{c} = \beta_\text{c} = 0.5$ and $\gamma_i = 1$, $ \forall i \in \mathcal{N}$.
}
\label{fig:1}
\end{figure}

\section{Conclusion and outlook} \label{sec:Concl}
A particular massaged implementation of the Douglas--Rachford operator splitting is applicable to generalized aggregative games for the computation of an aggregate equilibrium via a single-layer, semi-decentralized algorithm. In our numerical experience, the derived Douglas--Rachford algorithm outperforms forward-backward, i.e., projected-pseudogradient, algorithms in terms of convergence speed.

The convergence guarantees for our DR algorithm are limited by Assumption \ref{ass:E-MON}, i.e., monotonicity on the extended space, which does not hold in general. However, our numerical experience suggests that global convergence holds even if the assumption does not hold. Thus, it would be interesting to weaken Assumption \ref{ass:E-MON} with a monotonicity-type assumption on the original game. It would be also valuable to adapt our DR algorithm to generalized Nash equilibrium seeking.


\section*{Appendix}
\subsubsection*{Proof of Proposition \ref{prop:games}}
Let us introduce the pseudo-subdifferential mapping of the extended game in \eqref{eq:EG_1}--\eqref{eq:EG_2}
\begin{align*} 
\textstyle
\mc U:
\begin{bmatrix}
\bs x\\[.2em]
\bs y\\[.2em]
\sigma
\end{bmatrix}
 \mapsto
\begin{bmatrix}
\col \left( \left\{ \partial_{x_i} \left[ f_i \left( x_i, \sigma \right) + \iota_{C_i}(x_i,y_i) \right] \right\}_{i=1}^N \right) \\[.2em]
\col \left( \left\{ \partial_{y_i} \left[ f_i \left( x_i, \sigma \right)+ \iota_{C_i}(x_i,y_i) \right] \right\}_{i=1}^N \right) \\
\partial_{\sigma} f_{\text c}(\bs x,\sigma)
\end{bmatrix}
\end{align*} 
and the set of coupling constraints $\bs{\mc V }$, defined as
\begin{equation*}
\bs{\mc V}:=
\left\{
\col(\bs x, \bs y, \sigma) \, | \; NM \bs y \leq0, \; \sigma -M \bs x = 0
\right\}.
\end{equation*}
We recall that $\col(\bx^*,\by^*,\sigma^*)$ is a v-GNE of the extended game in \eqref{eq:EG_1}--\eqref{eq:EG_2} if and only if it is a solution to GVI$(\bs{\mc V}, \mc U)$.
It follows from \cite[Th. 3.1]{auslender:teboulle:00}, that $\col(\bx^*,\by^*,\sigma^*)$ solves GVI$(\bs{\mc V},\mc U)$ if and only if there exist $\lambda^*$ and $\mu^*$ such that the following $(N+1)$ sets of KKT conditions are satisfied:
\begin{align} \label{eq:KKT1-a}
\forall i\; &\begin{cases} 
0 \in \partial_{x_i} (f_i \left( x_i^*, \sigma^* \right) +  \iota_{C_i}(x_i^*, y_i^*)) -  \frac{1}{N}\mu^* \\
0 \in \partial_{y_i} \iota_{C_i}(x_i^*, y_i^*) + \lambda^* \\
0 = \sigma^* - M_n \bx^*\\
0  \in \nc_{\R^m_{\geq 0}}(\lambda^*) - NM_m \bs y^*
\end{cases}   
\\
\label{eq:KKT1-b}
&\begin{cases}
0 = \partial_\sigma f_\text{c}(\bs x^*, \sigma^*)+ \mu^*\\
0 =  \sigma^* - M_n \bx^*
\end{cases}
\end{align}
Note that \eqref{eq:KKT1-a}--\eqref{eq:KKT1-b} are satisfied if and only if $\sigma^* = M_n \bx^*$, $\mu^* = 0$ (since $\partial_\sigma f_\text{c}(\bs x^*, \sigma^*)=0$) and for all $i \in \mathcal{N}$
\begin{align} \label{eq:KKT2}
\begin{cases}
0 \in \partial_{x_i} f_i \left( x_i^*, \sigma^* \right)|_{\sigma^* = M \bs x^*} + \iota_{\mathcal{C}_i}(x_i^*, y_i^*) \\
0 \in \partial_{y_i}  \iota_{\mathcal{C}_i}(x_i^*, y_i^*) + \lambda^* \\
0  \in \nc_{\R^m_{\geq 0}}(\lambda^*) -NM \by^*
\end{cases} 
\end{align}
If we define $z_i = \col(x_i,y_i)$, then the first two inclusions in \eqref{eq:KKT2} can be recast in compact form as
\begin{equation*}
0 \in \partial_{z_i} \left[ 
f_i(x_i^*,\sigma^*)|_{\sigma^*=M\bs x^*} + \iota_{C_i}(x_i^*,y_i^*) + \lambda^\top y_i^*
\right].
\end{equation*}
Note that the last inclusion is satisfied if and only if $(x_i^*,y_i^*) \in C_i$, i.e., $y_i^* = A_i x_i^*-b_i$ and $x_i^* \in \Omega_i$, and $0 \in \partial_{x_i} \left[
f_i(x_i^*,\sigma^*)|_{\sigma^* = M\bs x^*} + \iota_{\Omega_i}(x_i^*) + \lambda^{*\top} A_i x_i^*
\right]$. Hence, $\col(\bx^*,\by^*,\lambda^*)$ satisfies the KKT conditions in \eqref{eq:KKT2} if and only if, for all $i \in \mathcal{N}$, $y_i^* = A_i x_i^*-b_i$ and
\begin{align}
\label{eq:KKT3}
\begin{cases}
0 \in \partial_{x_i} f_i \left( x_i^*, \sigma^* \right)|_{\sigma^* = M\bs x^*} + \nc_{\Omega_i}(x_i^*) + A_i^\top \lambda^* \\
0  \in \nc_{\R^m_{\geq 0}}(\lambda^*) -(A \bx^* -b)
\end{cases}
\end{align}
By \cite[Th. 3.1]{auslender:teboulle:00}, the pair $(x_i^*, \lambda^*)$ satisfy \eqref{eq:KKT3} for all $i \in \mathcal{N}$ if and only if $\bx^*$ solves GVI$(F_{\text a}, \bs{\mc X})$, with $F_{\text a}$ as in \eqref{eq:PsGrAgg} and $\bs{\mc X}$ as in \eqref{eq:G}, namely, $\bx^*$ is a v-GAE of the game in \eqref{eq:Game}.
{\hfill $\blacksquare$}

\subsubsection*{Proof of Proposition \ref{prop:opT-vGNE}}
The statement follows by noticing that $(x_i^*,y_i^*,\sigma^*,\mu^*,\lambda^*)$ satisfy the KKT in \eqref{eq:KKT1-a}--\eqref{eq:KKT1-b}, for all $i \in \mathcal{N}$, if and only if  $\col(\bx^*,\by^*,\sigma^*,\mu^*,\lambda^*) \in \zer \mathcal{T}$ .
{\hfill $\blacksquare$}

\subsubsection*{Proof of Lemma \ref{lemm:Splitting}}
The mapping $\cA$ is the sum of 2 terms: (1) $\mc A_1=\nc_{\bs{C}} \times \0_{d_2}$, which is maximally monotone since is the direct sum of maximally monotone operators \cite[Prop. 20.23]{bauschke2017convex}, namely, $\nc_{\bs{C}}$, normal cone of a closed convex set, thus maximally monotone, and $\0_{d_2}$, obviously maximally monotone; (2) $\mc \A_2={F}_{\text e} \times \0_{d_1}$ which is maximally monotone by Assumption \ref{ass:E-MON}. The maximal monotonicity of $\mc A_1 + \mc A_2 = \mc A$ follows from \cite[Cor. 24.4(i)]{bauschke2017convex}, since $\dom F_{\text e}\times \0_{d_1}  = \R^{d}$.
%
The mapping $\mathcal{B}$ is the sum of 2 terms: (1) $\mathcal{B}_1= \0_{d_3} \times \R^m_{\geq 0}$ which is maximally monotone for the same reasons of $\mc A_1$ and  (2) $S$ which is a linear, skew symmetric operator, thus maximally monotone \cite[Ex. 20.30]{bauschke2017convex}. Then, the maximal monotonicity of $\mc B_1 + \mc B_2 = \mc B$ follows from \cite[Cor. 24.4(i)]{bauschke2017convex} since $\dom \mathcal{B}_2 = \R^{d}$.
{\hfill $\blacksquare$}

\smallskip
\subsection*{Proof of Theorem \ref{th.Convergence}: }
The proof is divided in two parts: 
\begin{enumerate}[(1)]
\item We show that Algorithm 1 corresponds to the Douglas-Rachford splitting applied on the mappings $\Gamma \mc A$ and $\Gamma \mc B$, where $\Gamma$ is a diagonal positive definite matrix that characterizes the step sizes of the algorithm.
\item We show that the mappings $\Gamma \mc A$ and $\Gamma \mc B$ satisfy the assumptions of \cite[Th. 26.11]{bauschke2017convex}, which establishes global convergence to $\zer (\Gamma\mc A + \Gamma\mc B) = \zer(\mc A  +\mc B) = \zer(\mc T)$.
\end{enumerate}

\smallskip
\noindent
(1): The goal is to explicitly derive the iterations in \eqref{eq:DRsplit} for the mappings $\Gamma \mc A$ and $\Gamma \mc B$, where $\Gamma \succ 0 $ is defined as
\begin{align} \label{eq:GAMMA}
\Gamma := \textrm{blkdiag} (\bs \gamma \otimes I_n, \bs \gamma \otimes I_m, \alpha I_n,\beta I_n, \delta I_m),
\end{align} 
with $\bs \gamma = \diag(\gamma_1,\cdots,\gamma_N) $ and $ \alpha,\beta,\delta,\gamma_i \in \R_{>0}, \, \forall i \in \mc N$.

The next Lemma shows how to compute the implicit resolvents $\Res_{\Gamma \mathcal{A}}$ and $\Res_{\Gamma \mathcal{B}}$ in a semi-decentralized fashion.
\begin{lemma} \label{lem:resolvents}
For any positive definite matrix $\Gamma$ as in \eqref{eq:GAMMA}, $\Res_{\Gamma \mathcal{A}}={(\Id+\Gamma \mathcal{A})}^{-1}$ and $\Res_{\Gamma \mathcal{B}}={(\Id+\Gamma \mathcal{B})}^{-1}$ read as
\begin{enumerate}[(i)]
\item $\Res_{\Gamma \mathcal{A}}\left( \col( \bx, \by, \sigma, \mu, \lambda) \right) \mapsto \col( \bx^+, \by^+, \sigma ,\mu , \lambda),$ with 
\begin{equation} \label{eq:JgA}
\textstyle
(\forall i \in \mc N)  \;
\begin{cases}
x_i^+  = \underset{v \in \Omega_i}{\argmin} \;
f_i \left( z,\sigma \right)+ \frac{1}{2\gamma_i} \norm{v-x_i}^2\\[-.2em]
\hspace*{6.5em} +\frac{1}{2\gamma_i} \norm{A_i v - b_i - y_i}^2 \\[.3em]
y_i^+ = A_i x_i^+-b_i, 
\end{cases} 
\end{equation}
\item $\Res_{\Gamma \mathcal{B}} \left( \col( \bx, \by, \sigma, \mu, \lambda) \right) \mapsto \col(\bx^+,\by^+,\sigma^+,\mu^+,\lambda^+)$, define $P=\1^\top_N \otimes I_m$ and $\hat{\gamma} := \frac{1}{N} \sum_{j=1}^N \gamma_j$, then
\begin{align} \label{eq:JgB}
&\begin{cases}
\mu^+ = \frac{N}{(1+\beta \alpha )N + \beta \hat{\gamma}} \, (\mu + \beta(\sigma - M \bx) ) \\
\lambda^+ = {\proj}_{\R^m_{\geq 0}} \frac{1}{1+\delta N \hat{\gamma}  } \left(
\lambda  +\delta P \by 
\right)\\
\bx^+ = \bx + \bgamma M^\top \mu^+ \\
\by^+ = \by - \bgamma P^\top \lambda^+ \\
\sigma^+ = \sigma - \alpha \mu^+,
\end{cases}
\end{align}
{\hfill $\square$}
\end{enumerate}
\end{lemma}
\begin{proof}
(i) Let $\bomega = \col(\bx,\by,\sigma,\mu,\lambda)$, then 
\begin{align}
\Res_{\Gamma \mathcal{A}}(\bomega) &= (\Id+ \Gamma \mathcal{A})^{-1}(\bomega) = \bomega^+ \nonumber \\
&\Leftrightarrow \bomega \in (\Id + \Gamma \cA)(\bomega^+) \nonumber \\
&\Leftrightarrow  \mathbf{0} \in \bomega^+ + \Gamma \cA(\bomega^+) - \bomega
\label{eq:Res-Incl}
\end{align}
By expanding \eqref{eq:Res-Incl} we obtain a system of equations and inclusions, i.e., 
$\sigma^+ = \sigma$, $\mu^+=\mu$, $\lambda^+ = \lambda$ and for all $i \in \mathcal{N}$
\begin{equation} \label{eq:Res-A-Incl}
\begin{cases} \textstyle
0 \in x_i^+ - \gamma_i \partial_{x_i} (f_i \left( x_i^+,\sigma \right) +  \iota_{\cC_i}(x_i^+,y_i^+)) - x_i \\
0 \in y_i^+ + \gamma_i \partial_{y_i} \iota_{\cC_i}(x_i^+,y_i^+)-y_i.
\end{cases}
\end{equation}
If we define $z_i^+= \col(x_i^+,y_i^+)$, then \eqref{eq:Res-A-Incl} can be written as
\begin{equation*}
0 \in \partial_{z_i^+} \left[ 
\gamma_i f_i \left( x_i^+,\sigma \right) +\gamma_i \iota_{\cC_i}(x_i^+,y_i^+) + \tfrac{1}{2} \norm{z_i^+ - z_i}^2  
\right].
\end{equation*}
The latter inclusion is verified if and only if
\begin{align*}
& \textstyle
(x_i^+,y_i^+)  = \underset{
\begin{subarray}{c}
v \in \Omega_i\\
w = A_i v -b_i,
\end{subarray}
}{\argmin} \, \gamma_i f_i \left(v,\sigma \right) + \frac{1}{2} \norm{ \begin{bmatrix}
v-x_i \\ w-y_i
\end{bmatrix} }^2
\end{align*}
which is equivalent to \eqref{eq:JgA} and concludes the proof.\\
(ii) Consider the analogous of inclusion \eqref{eq:Res-Incl} for $\cB$, i.e.,
\begin{equation} \label{eq:Res-B-Sys}
\begin{cases}
\bx^+ - \boldsymbol{\gamma} M^\top \mu^+ - \bx &= 0 \\
\by^+ + \boldsymbol{\gamma} P^\top \lambda^+ - \by &= 0\\
\sigma^+ + \alpha \mu^+ - \sigma &= 0 \\
\mu^+ + \beta (M \bx^+ - \sigma^+) - \mu &=0 \\
\lambda^+ + \delta (\nc_{\R^m_{\geq 0 }}(\lambda^+) - P \by^+) - \lambda & \ni 0
\end{cases} 
\end{equation}
By substituting \eqref{eq:Res-B-Sys}b into \eqref{eq:Res-B-Sys}e and exploiting the equivalence $P \bgamma P^\top = \hat{\gamma} PP^\top$, we recast the latter inclusion as
\begin{multline} \textstyle \label{eq:lambdaPD}
\partial_{\lambda^+}\left[
\delta \iota_{\R^m_{\geq 0}}(\lambda^+) + \tfrac{1}{2} \delta \hat{\gamma} \norm{P^\top \lambda^+ - \frac{1}{\hat{\gamma}} y}^2 \right. \\
\left. +\tfrac{1}{2} \norm{\lambda^+ - \lambda }^2 \right] \ni 0.
\end{multline}
From \eqref{eq:lambdaPD}, it follows that
\begin{equation*} 
\lambda^+ = \underset{u \in \R^m_{\geq 0}}{\argmin} \; \frac{1}{2}\delta \hat{\gamma} \norm{ P^\top  u - \hat{\gamma}^{-1} \by}^2 +\frac{1}{2} \norm{u -\lambda}^2,
\end{equation*}
which is a quadratic problem, i.e.,
\begin{equation} \label{eq:lambdaQ} \textstyle
\lambda^+ = \underset{u \in \R^m_{\geq 0}}{\argmin} \; \frac{1}{2} u^{\top} Q u - c^\top u +d,  
\end{equation}
with $Q = \delta \hat{\gamma} P P^\top + I$, $c = \delta \by P^\top + \lambda $, $d = \delta^2 \bar{\gamma}^{-1} \by^\top \by + \lambda^\top \lambda$. The explicit solution to \eqref{eq:lambdaQ} reads as 
\begin{equation*}
\textstyle
\lambda^+ = {\proj}_{\R^m_{\geq 0}} Q^{-1}c = 
{\proj}_{\R^m_{\geq 0}} \frac{1}{1+\delta N \bar{\gamma}  } \left(
\lambda  +\delta P \by 
\right).
\end{equation*} 
The remaining equations in \eqref{eq:Res-B-Sys} can be solved by substitution, thus obtaining
\begin{equation}
\begin{cases}
\mu^+ = \frac{N}{(1+\beta \alpha )N + \beta \bar{\gamma}} \, (\mu + \beta(\sigma - M \bx) ) \\
\bx^+ = \bx + \bgamma M^\top \mu^+ \\
\by^+ = \by - \bgamma P^\top \lambda^+ \\
\sigma^+ = \sigma - \alpha \mu^+ 
\end{cases}
\end{equation}
\end{proof}

\smallskip
By exploiting \eqref{eq:JgA} and \eqref{eq:JgB}, we explicitly write \eqref{eq:DRsplit} for the mappings $\Gamma \mc A$ and $\Gamma \mc B$, with $\lambda_{k} = 1$ for all $k \in \mathbb{N}$, i.e.,
\begin{align*}
\text{(a)}& \quad \textstyle
\begin{cases}
x_i^{k+1/2}  = \underset{z \in \Omega_i}{\argmin} \;
f_i \left( z,\tilde{\sigma}^{k} \right) + \frac{1}{2\gamma_i} \norm{z-\tilde{x}_i^k}^2
\\[-.5em] \hspace*{2.6cm} + \frac{1}{2\gamma_i}  \norm{A_i z- b_i - \tilde{y}_i^k}^2, \; \forall i \in \Ncal  \\
y_i^{k+1/2} = A_i x_i^{k+1/2}-b_i, \quad \forall i \in \Ncal \\
\sigma^{k+1/2} = \tilde{\sigma}^k\\
\mu^{k+1/2} = \tilde{\mu}^k\\
\lambda^{k+1/2} = \tilde{\lambda}^k
\end{cases} \\[.5em]
\text{(b)}& \quad \textstyle
\begin{cases}
\tilde{x}_i^{k+1/2} = 2 x_i^{k+1/2}- \tilde{x}_i^k, \quad \forall i \in \Ncal \\
\tilde{y}_i^{k+1/2} = 2 A_i x_i^{k+1/2} - \tilde{y}_i^k, \quad \forall i \in \Ncal \\
\tilde{\sigma}^{k+1/2} = \tilde{\sigma}^k\\
\tilde{\mu}^{k+1/2} = \tilde{\mu}^k\\
\tilde{\lambda}^{k+1/2} = \tilde{\lambda}^k
\end{cases}
\end{align*}
\begin{align*}
\text{(c)}& \quad \textstyle
\begin{cases}
x_i^{k+1} = \tilde{x}_i^{k+1/2} + \gamma_i \frac{1}{N} \mu^{k+1}, \quad \forall i \in \Ncal  \\
y_i^{k+1} = \tilde{y}_i^{k+1/2} - \gamma_i \lambda^{k+1}, \quad \forall i \in \Ncal \\
\sigma^{k+1} = \tilde{\sigma}^k - \alpha \mu^{k+1}\\
\mu^{k+1} = \frac{N}{(1+\beta \alpha )N + \beta \hat{\gamma}} \, (\tilde{\mu}^k + \beta( \tilde{\sigma}^k - M \tilde{\bx}^{k+1/2})) \\
\lambda^{k+1}={\proj}_{\R^m_{\geq 0}} \frac{1}{1+\delta N \hat{\gamma}  } \left(
\tilde{\lambda}^{k}  +\delta P \tilde{\by}^{k+1/2} 
\right)
\end{cases} \\[.5em]
\text{(d)}& \quad \textstyle
\begin{cases}
\tilde{x}^{k+1}_i = x_i^{k+1/2} + \gamma_i \frac{1}{N} \mu^{k+1},  \quad \forall i \in \Ncal \\
\tilde{y}_i^{k+1} = y_i^{k+1/2} - \gamma_i \lambda^{k+1}, \quad \forall i \in \Ncal \\
\tilde{\sigma}^{k+1} = \sigma^{k+1} \\
\tilde{\mu}^{k+1} = \mu^{k+1}\\
\tilde{\lambda}^{k+1} = \lambda^{k+1}
\end{cases}
\end{align*}

Finally, by sorting and simplifying (a)--(d), and setting
\begin{align} \label{eq:deltaS}
\delta_\text{c} &:=\textstyle \frac{\delta}{\delta \hat{\gamma}+\frac{1}{N}}, & \delta \in \R_{>0} &\Rightarrow \delta_\text{c} \in (0,\hat{\gamma}^{-1}), \\
\textstyle \label{eq:betas}
\beta_\text{c} &:=\textstyle \frac{\beta}{1+\beta(\alpha+ \hat{\gamma}/N)},&
 \beta \in \R_{>0} & \textstyle \Rightarrow \beta_\text{c} \in (0,\frac{1}{\alpha + \hat{\gamma}/N}) ,
\end{align}
we obtain the iterations in Algorithm 1.

\smallskip
(2): The mappings $\mc A$ and $\mc B$ are maximally monotone, by Lemma \ref{lemm:Splitting}, and $\zer(\mc A+ \mc B) \neq \emptyset$, by Remark \ref{rem:Existence}.
Since $\Gamma= \Gamma^\top \succ 0$, then $\Gamma\mc A$ and $\Gamma\mc B$ are maximally monotone in the space defined by the norm $\norm{\cdot}_{\Gamma}$ and
$\zer (\Gamma\mc A + \Gamma\mc B) = \zer(\mc A  +\mc B) \neq \emptyset$. Therefore, we can apply \cite[Th. 26.11]{bauschke2017convex} to establish the global convergence of the sequence $\left( \col(\bs x^k, \bs y^k, \sigma^k, \mu^k, \lambda^k)    \right)_{k=0}^\infty$ generated by Algorithm 1, to some $ \col(\bs x^*, \bs y^*, \sigma^*, \mu^*, \lambda^*) \in \zer(  \mc A + \mc B) = \zer(\mc T)$. By Prop. \ref{prop:opT-vGNE}, $\bs x^*$ is a v-GNE of the extended game in \eqref{eq:EG_1}-\eqref{eq:EG_2}. To conclude, we recall Prop. \ref{prop:games} to show that $\bs x^*$ is also a v-GAE of the original game in \eqref{eq:Game}.
{\hfill $\blacksquare$}

\balance
\bibliographystyle{IEEEtran}
\bibliography{library}

\end{document}